\documentclass[10pt]{amsart}

\usepackage{amsmath,amssymb,latexsym,cite,mathrsfs}
\usepackage[utf8]{inputenc}
\usepackage{color,enumitem,graphicx}
\usepackage{nameref}
\usepackage{varioref}
\usepackage[colorlinks=true,urlcolor=blue,
citecolor=red,linkcolor=blue,linktocpage,pdfpagelabels,
bookmarksnumbered,bookmarksopen]{hyperref}
\usepackage{cleveref}
\usepackage[english]{babel}
\usepackage{verbatim}
\usepackage[left=2.9cm,right=2.9cm,top=2.8cm,bottom=2.8cm]{geometry}
\usepackage[hyperpageref]{backref}
\usepackage[pagewise]{lineno}%\linenumbers

\usepackage[colorinlistoftodos]{todonotes}
\makeatletter
\providecommand\@dotsep{5}
\def\listtodoname{List of Todos}
\def\listoftodos{\@starttoc{tdo}\listtodoname}
\makeatother

\numberwithin{equation}{section}
\everymath{\displaystyle}

\newtheorem{theorem}{Theorem}[section]
\newtheorem{proposition}[theorem]{Proposition}
\newtheorem{lemma}[theorem]{Lemma}
\newtheorem{corollary}[theorem]{Corollary}

\newtheorem{remark}{Remark}

\newcommand{\R}{\mathbb{R}}
\newcommand{\N}{\mathbb{N}}

\begin{document}

\title[Nonlocal logistic system with nonlinear advection terms]{On the existence of a positive solution to a nonlocal logistic system with nonlinear advection terms}
\author{Willian Cintra, Romildo Lima and Mayra Soares}

\address[Willian Cintra and Mayra Soares]
{\newline\indent Departamento de Matemática
	\newline\indent 
	Universidade de Brasília
	\newline\indent
	70910-900, Brasília - DF, Brazil} 
\email{\href{mailto: willian@unb.br; mayra.soares@unb.br}{willian@unb.br; mayra.soares@unb.br}}

\address[Romildo Lima]
{\newline\indent Unidade Acadêmica de Matemática
\newline\indent 
Universidade Federal de Campina Grande 
\newline\indent
58429-970, Campina Grande - PB, Brazil} 
\email{\href{mailto: romildo@mat.ufcg.edu.br}{romildo@mat.ufcg.edu.br}}

\pretolerance10000

%\begin{document}

\begin{abstract} 
In this paper, we study a nonlocal logistic system with nonlinear advection terms
\begin{equation*}
    \left\{
\begin{array}{lcl}
-\Delta u+\vec{\alpha}(x)\cdot \nabla (|u|^{p-1}u)&=&\left(a-\int_{\Omega}K_1(x,y)f(u,v)dy \right)u+bv\mbox{ in }\Omega,\\
-\Delta v+\vec{\beta}(x)\cdot \nabla (|v|^{q-1}v)&=&\left(d-\int_{\Omega}K_2(x,y)g(u,v)dy \right)v+cu\mbox{ in }\Omega,\\
\qquad \qquad \qquad \qquad u=v&=&0\mbox{ on }\partial\Omega,
\end{array}
\right.
\end{equation*}
where $\Omega\subset\R^N$, $N\geq1$, is a bounded domain with a smooth boundary, $\vec{\alpha}(x)=(\alpha_1(x),\cdots,\alpha_N(x))$ and $\vec{\beta}(x)=(\beta_1(x),\cdots,\beta_N(x))$  are flows satisfying suitable conditions, $p,q\geq1$, $a,b,c,d>0$ and $K_1,K_2:\Omega\times\Omega\rightarrow\R$ are nonnegative functions, with their specific conditions detailed below. The functions $f$ and $g$ satisfy some assumptions which allow  us to use bifurcation theory to prove the existence of solution to problem $(P)$. It is important to highlight that the inclusion of the integral nonlocal term on the right-hand side makes the problem more representative of real-world situations. 
\end{abstract}

\thanks{R. Lima was partially supported by CNPq/Brazil 306.411/2022-9 and 176.596/2023-2.}
\thanks{W. Cintra was partially supported by FAPDF 00193.00001821/2022-21 and CNPq 310664/2023-3}

\subjclass[2020]{35B50, 35B51, 35J60} 
\keywords{Maximum principles in context of PDE's, Comparison principles in context of PDE's, Nonlinear elliptic equations}

\maketitle
\tableofcontents
\section{Introduction and Main Results}
In this work, we study a nonlocal logistic system with nonlinear advection terms
\begin{equation}\tag{P}\label{P}
    \left\{
\begin{array}{lcl}
-\Delta u+\vec{\alpha}(x)\cdot \nabla (|u|^{p-1}u)&=&\left(a-\int_{\Omega}K_1(x,y)f(u,v)dy \right)u+bv\mbox{ in }\Omega,\\
-\Delta v+\vec{\beta}(x)\cdot \nabla (|v|^{q-1}v)&=&\left(d-\int_{\Omega}K_2(x,y)g(u,v)dy \right)v+cu\mbox{ in }\Omega,\\
\qquad \qquad \qquad \qquad u=v&=&0\mbox{ on }\partial\Omega,
\end{array}
\right.
\end{equation}
where $\Omega\subset\R^N$, $N\geq1$, is a bounded domain with smooth boundary, $\vec{\alpha}(x)=(\alpha_1(x),\cdots,\alpha_N(x))$ and $\vec{\beta}(x)=(\beta_1(x),\cdots,\beta_N(x))$  are flows satisfying suitable conditions, $p,q\geq1$ and $a,b,c,d>0$. Moreover, $K_1,K_2:\Omega\times\Omega\rightarrow\R$ are nonnegative functions. The functions $f$ and $g$ satisfy some assumptions which allow us to use bifurcation theory to prove the existence of solution to \eqref{P}. 

This system can be interpreted as a reaction-diffusion-advection model within the framework of population dynamics. In this model, the domain $\Omega$ represents the habitat where two species coexist, with their population densities denoted by $u(x)$ and $v(x)$, respectively. The left-hand side of the equations  includes random diffusion terms, characterized by the Laplacian operator, which model the random movement of individuals, and nonlinear advection movement terms, represented by first-order derivatives, accounting for directed movement influenced by factors such as gradients in resources or population density. The right-hand side incorporates reaction terms that depend not only on local population densities but also on nonlocal interactions, reflecting the influence of individuals at a distance. We consider homogeneous Dirichlet boundary conditions, which imply that the boundary is a lethal zone for both species.

The focus of our work is the class of logistic systems as highlighted, however the reader interested in the study of logistic equations can check the works of Allegretto and Nistri \cite{Allegretto-Nistri}, Belgacem and Cosner \cite{BelgacemCosner}, Cantrell and Cosner \cite{CantrellCosner,CantrellCosner2}, Chen and Shi \cite{Chen-Shi}, Chipot \cite{Chipot}, Corrêa, Delgado and Suarez \cite{Correa-Delgado-Suarez}, Cosner \cite{Cosner}, Coville \cite{Coville}, Sun, Shi and Wang \cite{Sun-Shi-Wang}, Umezu \cite{U} and their references.

In a recent paper by Lima, Duarte, and Souto \cite{R}, the equation corresponding to problem \eqref{P} was investigated. Specifically, they analyzed the existence of positive solutions for the following equation:
$$
\left\{
\begin{array}{lcl}
	-\Delta u+\vec{\alpha}(x)\cdot \nabla (|u|^{p-1}u)&=&\left(\lambda-\int_{\Omega}K(x,y)u^{\gamma}(y)dy \right)u\mbox{ in }\Omega,\\
 \qquad \qquad \qquad \qquad \qquad u&=&0\mbox{ on }\partial\Omega.
\end{array}
\right.
$$
This model combines the nonlinear advection term introduced in \cite{CintraMontenegroSuarez} with the nonlocal reaction term from \cite{Alves-Delgado-Souto-Suarez}.

This work inspired us to consider the associated system, which presents several technical challenges due to the simultaneous presence of the nonlinear advection term and the nonlocal character.

Indeed, the study of such models via bifurcation theory remains a challenging problem. For instance, it is well-known in the literature that quasilinear elliptic systems are particularly challenging, especially when they cannot be transformed into a semilinear system via a change of variables, as is the case for \eqref{P}. To analyze such problems from the bifurcation theory perspective, the theorems of Shi and Wang \cite{SW} and Cintra, Morales-Rodrigues, and Suárez \cite{CMRS2019,CMS} have been useful tools.  A different approach is employed by Lima and Souto in \cite{RM} (see also \cite{CS1}), where the authors apply classical global bifurcation results to obtain a continuum of positive solutions emanating from the trivial solution for a nonlocal elliptic system.

Inspired by these papers, in this work, we consider equation \eqref{P}, which combines the technical challenges of the nonlocal term and the nonlinear advection.  

Throughout this work, we will adopt the following assumptions:

\begin{itemize}
    \item $K_1$ and $K_2$ belonging to the class of functions $\mathcal{K}$, which consists of functions $K:\Omega\times\Omega\rightarrow\R$ satisfying
\begin{itemize}
\item[$(K_1)$] $K\in L^{\infty}(\Omega\times\Omega)$ and $K(x,y)\geq0$ a. e. in  $\Omega\times \Omega$.
\item[$(K_2)$]  There exist $\varepsilon>0$ such that 
$$
	|[B_{\varepsilon}(x)\times B_{\varepsilon}(x)]\cap K_+|>0 \quad \forall x \in \Omega,
$$
where $K_+:=\{(x,y)\in\Omega\times\Omega \ : \ K(x,y)>0 \}=K^{-1}(0,+\infty)$.
\end{itemize}

\item The vectorial functions $\vec{\alpha},\ \vec{\beta}:\overline{\Omega}:\to \mathbb{R}^N$ are of class $ C^1$

\item  Regarding the functions $f$ and $g$, as in \cite{RM}, we assume that they satisfy
\begin{itemize}
\item[$(f_0)$] $f,g:[0,\infty)\times[0,\infty)\rightarrow\R^+$ are continuous functions;

\item[$(f_1)$] There exists $\varepsilon_0>0$ such that $f(t,s)\geq\varepsilon_0|t|^\gamma$ and $g(t,s)\geq\varepsilon_0|s|^\gamma$, for all $t,s\in[0,\infty)$ and $\gamma>0$;

\item[$(f_2)$] $f(\xi t,\xi s)=\xi^{\gamma}f(t,s)$ and $g(\xi t,\xi s)=\xi^{\gamma}g(t,s)$, for all $t,s\in[0,\infty)$ and $\xi>0$, where $\gamma>0$.
\end{itemize}

\end{itemize}

The functions $f(t,s) = |t|^{\gamma} + |s|^{\gamma-\mu}|t|^{\mu}$ and $g(t,s) = c_1|t|^{\gamma} + c_2|s|^{\gamma}$ are examples that satisfy conditions $(f_0)-(f_2)$.

\begin{remark}\label{remar1}
The assumption $(K_2)$ is equivalent to assuming that: If $w$ is measurable and $\int_{\Omega\times\Omega}K(x,y)|w(y)|^{\gamma}|w(x)|^{2}dxdy=0$, then $w=0$ a.e. in $\Omega$.
\end{remark}

Furthermore, by result of linear algebra, we can assume that the matrix
\[
A = \begin{pmatrix} a & b \\ c & d \end{pmatrix}
\]  
has \(\lambda_A > 0\) the largest eigenvalue of the matrix. It is well know that there exists an eigenvector $z=(z_1,z_2)$ of $A$ associated to \(\lambda_A \) with $z_1,z_2>0$.
  
The principal eigenvalue of \(-\Delta + \vec{\alpha} \cdot \nabla\) in \(\Omega\) with homogeneous Dirichlet boundary conditions will be denoted by \(\lambda_1^\alpha\). When \(\vec{\alpha} = 0\), this eigenvalue will be denoted by \(\lambda_1\).

Using global and local bifurcation results, we will establish existence and nonexistence results for positive classical solutions of \eqref{P}, that is, a pair $(u,v) \in C^2_0(\overline{\Omega})\times C^2_0(\overline{\Omega})$ with both components positive.

Now we can state our main results regarding the existence of a positive solution for problem \eqref{P}. The first result concerns the case where \( p = q = 1 \).

\begin{theorem} \label{TP1}
	Assume that $K_1,K_2\in \mathcal{K}$ and $(f_0)-(f_2)$ hold. If $p=q=1$, $\vec{\alpha} = \vec{\beta}$  and $\operatorname{div}\vec{\alpha}(x)=0$ in $\Omega$ then problem \eqref{P} 	has a positive solution if and only if $\lambda_A>\lambda_1^{\alpha}$.
%	$$
%	\left\{
%	\begin{array}{lcl}
%		-\Delta u+\vec{\alpha}(x)\cdot \nabla u&=&\left(a-\int_{\Omega}K_1(x,y)f(u,v)dy \right)u+bv\mbox{ in }\Omega,\\
%		-\Delta v+\vec{\alpha}(x)\cdot \nabla v&=&\left(d-\int_{\Omega}K_2(x,y)g(u,v)dy \right)v+cu\mbox{ in }\Omega,\\
%	\quad \qquad \qquad	u=v&=&0\mbox{ on }\partial\Omega,
%	\end{array}
%	\right.\leqno(P_{1,\alpha})
%	$$
\end{theorem}

The second result deals with the system where \( p, q > 1 \).

\begin{theorem} \label{TP2}
Assume that $K_1,K_2\in \mathcal{K}$ and $(f_0)-(f_2)$ hold with  $\gamma>\max\{p,q\}$.  If $p,q>1$ and $\operatorname{div}\vec{\alpha}(x)=\operatorname{div}\vec{\beta}(x)=0$ in $\Omega$ then problem \eqref{P} has positive solution if, and only if $\lambda_A>\lambda_1$.
%$$
%\left\{
%\begin{array}{lcl}
%-\Delta u+\vec{\alpha}(x)\cdot \nabla (|u|^{p-1}u)&=&\left(a-\int_{\Omega}K_1(x,y)f(u,v)dy \right)u+bv\mbox{ in }\Omega,\\
%-\Delta v+\vec{\beta}(x)\cdot \nabla (|v|^{q-1}v)&=&\left(d-\int_{\Omega}K_2(x,y)g(u,v)dy \right)v+bu\mbox{ in }\Omega,\\
%\qquad \qquad \qquad \qquad u=v&=&0\mbox{ on }\partial\Omega,
%\end{array}
%\right.\leqno(P_{p,q})
%$$
\end{theorem}

The condition on the divergence of the vector fields \( \vec{\alpha} \) and \( \vec{\beta} \) ensures that there is no positive solution to \eqref{P} for \( \lambda_A \leq \lambda_1^\alpha \) or \( \lambda_A \leq \lambda_1 \) (see Lemma \ref{lem:nosolution} and \ref{lem:nosolution2}). In Section 5, we will use local bifurcation results to show that, without this condition, in some cases, there exists a positive solution to \eqref{P} for \( \lambda_A \) smaller than \( \lambda_1 \).

\begin{remark}\label{remar2}
The solutions given by Theorems \ref{TP1} and \ref{TP2} belong to \(C^1_0(\overline{\Omega})\). Therefore, assuming that \(\operatorname{div}\vec{\alpha}(x) = 0\), the identity 
\begin{equation*}
p\int_{\Omega}u^p[\vec{\alpha}(x)\cdot\nabla u] dx=\frac{p}{p+1}\int_{\Omega}\vec{\alpha}(x)\cdot \nabla u^{p+1}dx=-\frac{p}{p+1}\int_{\Omega}u^{p+1}\operatorname{div}(\vec{\alpha}(x))dx+\frac{p}{p+1}\int_{\partial\Omega} u^{p+1}(\vec{\alpha}\cdot\eta) d\sigma,
\end{equation*}
which holds for all \(u \in C^1(\overline{\Omega})\) and \(\vec{\alpha} \in C^1(\overline{\Omega})\), is fundamental and will be used at various points throughout this paper.
\end{remark}

At this stage, it is crucial to emphasize that the inspiration for the results presented here stems primarily from the works \cite{R} and \cite{RM}. Several key findings from these studies are indispensable to our analysis. However, to establish our results, particularly Theorem \ref{TP2}, we introduce a set of "soft" assumptions. A good example of how "soft" our hypotheses are is the fact that symmetry in $A$ is not required. Despite this, we still obtain the non-existence result, as demonstrated in Lemmas \ref{lem:nosolution} and \ref{lem:nosolution2}. We believe that, while these assumptions may initially appear unconventional due to the non-linearities involved, we prove to be sufficient. Additionally, we conduct a detailed study of the behavior of the branch of positive solutions to the problem, see Section 5, a significant aspect for real-world applications that was not addressed in the works that inspired our research.

The paper is organized as follows. Section 2 introduces preliminary results and definitions that are essential for achieving our objectives. 
In Section 3, we establish the existence and nonexistence of positive solutions to problem \eqref{P} with \(p = q = 1\) and \(\vec{\alpha} = \vec{\beta}\). Specifically, we prove Theorem \ref{TP1}.
In Section 4, we establish the existence and nonexistence of positive solutions to problem \eqref{P} with \(p, q > 1\). In particular, we prove Theorem \ref{TP2}.
Finally, Section 5 provides a brief analysis of the behavior of the positive solution branches obtained in the previous theorems.

\section{Preliminaries}

In this section, we construct the operators essential for the development of results in subsequent sections. Firstly, we fix some notations:
\begin{itemize}
	\item[(i)] For every $\vec{\alpha}=(\alpha_1,\cdots,\alpha_N)\in\R^N$, we denote the euclidian norm by $|\vec{\alpha}|=\sqrt{\alpha_1^2+\cdots+\alpha_N^2}$; 
	\item[(ii)] $L^s(\Omega)$, for $1\leq s\leq\infty$ is the Lebesgue space with the usual norm denoted by $|u|_s$.
	\item[(iii)] $\|\cdot\|_{C^j(\overline{\Omega})}$, $\|\cdot\|$ and $\|\cdot\|_{p,q}$ denote the usual norms of the spaces $C^j(\overline{\Omega})$, for $j\in\N\cup\{0\}$, the Sobolev space $H_0^1(\Omega)$ and the Sobolev  space $W^{p,q}(\Omega)$, respectively.
	
	\item[(iv)] $\sigma(A)$ denotes the real spectrum of the matrix $A$. Similarly, $\sigma(-\Delta + \vec{\alpha} \cdot \nabla)$ represents the spectrum of the operator $(-\Delta+\vec{\alpha} \cdot \nabla, H_{0}^{1}(\Omega))$.
	
	\item[(v)]  The terms of the form $U = (u,v)$ will be written as a column matrix
	$U=\left(\begin{array}{c}u\\v\end{array}\right)$,  whenever convenient. Moreover, $-\Delta U=(-\Delta u,-\Delta v)$ or $-\Delta U=\left(\begin{array}{c}-\Delta u\\-\Delta v\end{array}\right)$ and 	
	$z=(\alpha,\beta)>0$ or $z=\left(\begin{array}{c}\alpha\\\beta\end{array}\right)>0$ means $\alpha,\beta>0$. 	
    \item[(vi)]  For each $U=(u,v) \in H_0^1(\Omega)\times H_0^1(\Omega)$, we define
\begin{equation*}
L(U):=\left[
\begin{array}{c}
-\Delta u+\vec{\alpha}(x)\cdot\nabla u\\
-\Delta v+\vec{\alpha}(x)\cdot\nabla v
\end{array}
\right].
\end{equation*}

	\item[(viii)] $E:=C(\overline{\Omega})\times C(\overline{\Omega})$, with the norm given by
	$\|U\|=\|u\|_{C^0(\overline{\Omega})}+\|v\|_{C^0(\overline{\Omega})}$,
	for all $U=(u,v)\in E$ and $E_1:=C^{1}_0(\overline{\Omega})\times C^1_0(\overline{\Omega})$, with the norm given by
	$\|U\|_1=\|u\|_{C^1(\overline{\Omega})}+\|v\|_{C^1(\overline{\Omega})},$ 	for all $U=(u,v)\in E_1$.
\end{itemize}

We now present the matrix formulation of the problem and the solution operators used in the main theorem proofs.
Since $K_1,K_2\in\mathcal{K}$, assuming $(f_0)-(f_2)$ and defining
$\Phi,\Psi:L^{\infty}(\Omega)\times L^{\infty}(\Omega)\rightarrow L^{\infty}(\Omega)$ by
$$\Phi_{(u,v)}(x)=\int_{\Omega}K(x,y)f(|u(y)|,|v(y)|)dy
\mbox{ \ and \ }
\Psi_{(u,v)}(x)=\int_{\Omega}K_2(x,y)g(|u(y)|,|v(y)|)dy,$$
the following conditions hold for $\Phi$ and $\Psi$:
\begin{itemize}
\item[$(\Phi_1)$] $t^\gamma \Phi_{(u,v)}=\Phi_{(tu,tv)}$ and $t^\gamma \Psi_{(u,v)}=\Psi_{(tu,tv)}$, for all $u,v\in L^{\infty}(\Omega)$ and $t>0$;
\item[$(\Phi_2)$] $\|\Phi_{(u,v)}\|_\infty\leq\|K\|_\infty|\Omega|\|f(|u|,|v|)\|_\infty$ and $\|\Psi_{(u,v)}\|_\infty\leq\|K\|_\infty|\Omega|\|g(|u|,|v|)\|_\infty$, for all $u,v\in L^{\infty}(\Omega)$.
\end{itemize}
So, considering positive solutions only, the problem \eqref{P} can be written in the form
$$
\left\{
\begin{array}{lcl}
	-\Delta u+\vec{\alpha}(x)\cdot \nabla (|u|^{p-1}u)+\Phi_{(u,v)}u&=&au+bv\quad\mbox{in}\quad\Omega,\\
	-\Delta v+\vec{\beta}(x)\cdot \nabla (|v|^{q-1}v)+\Psi_{(u,v)}u&=&cu+dv\quad\mbox{in}\quad\Omega,\\
 \qquad \qquad \qquad \qquad \qquad \qquad	u,v&>&0\quad\mbox{in}\quad\Omega,\\
\qquad \qquad \qquad \qquad \qquad \qquad	u=v&=&0\quad\mbox{on}\quad\partial\Omega.
\end{array}
\right.\leqno(P^+)
$$
Here, we point out that $U=(u,v)$ satisfies the above problem in the weak sense, if $u,v\in H_{0}^{1}(\Omega)$ and
\begin{eqnarray}
	\int_{\Omega}\nabla u\nabla\varphi dx+\int_{\Omega}p|u|^{p-1}(\vec{\alpha}(x)\cdot\nabla u)\varphi dx+\int_{\Omega}\Phi_{(u,v)}(x)u\varphi dx=\int_{\Omega}(au+bv)\varphi dx\\
	\int_{\Omega}\nabla v\nabla\eta dx+\int_{\Omega}q|v|^{q-1}(\vec{\beta}(x)\cdot\nabla v)\eta dx+\int_{\Omega}\Psi_{(u,v)}(x)v\eta dx=\int_{\Omega}(cu+dv)\eta dx
\end{eqnarray}
for all $\eta,\varphi\in H_{0}^{1}(\Omega)$.

We will prove the existence of a positive solution to problem \eqref{P} using the  global bifurcation result due to Rabinowitz, see \cite{Rabinowitz}. Firstly, we recall that, if $\vec{\mu}\in C(\overline{\Omega})^N$,  for each $h\in L^{\infty}(\Omega)$, there exists a unique function $\omega\in C^{1}(\overline{\Omega})$ satisfying
$$
\left\{\begin{array}{lcl}
	-\Delta\omega+\vec{\mu}\cdot \nabla \omega&=&h(x)\mbox{ in }\Omega,\\
\qquad \qquad \quad	\omega&=&0\mbox{ on }\partial\Omega.
\end{array}
\right.
$$
Moreover, there exists $c_\infty=c_\infty(\Omega,\vec{\mu})>0$ such that $\|\omega\|_{C^{1}(\overline{\Omega})}\leq c_\infty\|h\|_\infty.$
We use this property freely in the construction and the elementary properties of the solution operators  below.

\noindent\textbf{The case $p=q=1$ and $\vec{\alpha}=\vec{\beta}$}.
We will consider the solution operator $S_1:E_1\rightarrow E_1$, given by
$$S_1(u,v)=(u_1,v_1)\Leftrightarrow\left\{\begin{array}{lcl}
	-\Delta u_1+\vec{\alpha}(x)\cdot\nabla u_1&=&au+bv\quad\mbox{in}\quad\Omega,\\
	-\Delta v_1+\vec{\alpha}(x)\cdot\nabla v_1&=&cu+dv\quad\mbox{in}\quad\Omega,\\
 \quad \qquad \qquad	u_1= v_1&=&0\quad\mbox{on}\quad\partial\Omega,
\end{array}\right.$$
In view of the Schauder embedding, we know that $S_1$ is well defined, linear, compact, and verifies
$$\|S_1(U)\|_1\leq C\|U\|_1,\mbox{ for all }U\in E_1\mbox{ and some }C>0,$$
where $U=(u,v)$ and $S_1(U)=U_1=(u_1,v_1)$. Similarly, we set the nonlinear operator $G_1:E_1\rightarrow E_1$ given by
$$G_1(u,v)=(u_1,v_1)\Leftrightarrow\left\{\begin{array}{lcl}
-\Delta u_1+\vec{\alpha}(x)\cdot\nabla u_1&=&-\Phi_{(u,v)}(x)u\quad\mbox{in}\quad\Omega,\\
-\Delta v_1+\vec{\alpha}(x)\cdot\nabla v_1&=&-\Psi_{(u,v)}(x)v\quad\mbox{in}\quad\Omega,\\
\quad \qquad \qquad u_1=v_1&=&0\quad\mbox{on}\quad\partial\Omega.
\end{array}\right.$$
Then, $G$ is well defined, continuous and satisfies
$$\|G_1(U)\|_1\leq C(\|\Phi_{(u,v)}\|_\infty+\|\Psi_{(u,v)}\|_\infty)\|U\|_1,\mbox{ for all }U\in E_1.$$
Using again the Schauder embedding, we have that $G_1:E_1\rightarrow E_1$ is compact. Additionally, using $(f_2)$ and $(\Phi_2)$, we verify that
$$G_1(U)=o(\|U\|_1).$$

\noindent\textbf{The case $p,q>1$}.
Consider the solution operator $S_{pq}:E_1\rightarrow E_1$, given by
$$S_{pq}(u,v)=(u_1,v_1)\Leftrightarrow\left\{\begin{array}{lcl}
	-\Delta u_1&=&au+bv\quad\mbox{in}\quad\Omega,\\
	-\Delta v_1&=&cu+dv\quad\mbox{in}\quad\Omega,\\
	u_1=v_1&=&0\quad\mbox{on}\quad\partial\Omega.
\end{array}\right.$$
We have that $S_{pq}$ is well defined, linear, compact and verifies
$$\|S_{pq}(U)\|_1\leq C\|U\|_1,\mbox{ for all }U\in E_1\mbox{ and some }C>0, 
\mbox{ where } U=(u,v) \mbox{ and } S_{pq}(U)=U_1=(u_1,v_1).$$

On the other hand, setting the nonlinear operator $G_{pq}:E_1\rightarrow E_1$ given by
$$G_{pq}(u,v)=(u_1,v_1)\Leftrightarrow\left\{\begin{array}{lcl}
	-\Delta u_1+p|u|^{p-1}\vec{\alpha}(x)\cdot\nabla u+\Phi_{(u,v)}(x)u&=&0\quad\mbox{in}\quad\Omega,\\
	-\Delta v_1+q|v|^{q-1}\vec{\beta}(x)\cdot\nabla v+\Psi_{(u,v)}(x)v&=&0\quad\mbox{in}\quad\Omega,\\
\qquad \qquad \quad \qquad \qquad \qquad \qquad	u_1=v_1&=&0\quad\mbox{on}\quad\partial\Omega.
\end{array}\right.$$
we have that $G_{pq}$ is well defined, continuous and checks
$$\|G_{pq}(U)\|_1\leq C(\|\Phi_{(u,v)}\|_\infty+c\|u\|_{\infty}^{p-1}+\|\Psi_{(u,v)}\|_\infty+c\|v\|_{\infty}^{q-1})\|U\|_1,\mbox{ for all }U\in E_1.$$
Using again the Schauder embedding, we have that $G_{pq}:E_1\rightarrow E_1$ is compact. Moreover, from $(f_2)$ and $(\Phi_2)$, it is possible to verify that
$$G_{pq}(U)=o(\|U\|_1).$$

It is important to highlight that some results obtained by Lima and Souto in \cite{RM} are also valid in our case and will be applied here. In order to be more precise, throughout this work we will cite these corresponding results as they become necessary to be used, although, we will not state or prove them again.

\section{Existence and nonexistence of solution to $p=q=1$ and $\vec{\alpha}=\vec{\beta}$}

We aim to prove Theorem \ref{TP1} via bifurcation theory, then we need to introduce a parameter $t\in \mathbb{R}$ into \eqref{P} and consider the following auxiliary problem

	\begin{equation}
	   	\left\{
	\begin{array}{lcl}
		-\Delta u+\vec{\alpha}(x)\cdot \nabla u+\Phi_{(u,v)}u&=&t(au+bv)\quad\mbox{in}\quad\Omega,\\
		-\Delta v+\vec{\alpha}(x)\cdot \nabla v+\Psi_{(u,v)}v&=&t(cu+dv)\quad\mbox{in}\quad\Omega,\\
\quad	\qquad \qquad \qquad \qquad	u=v&=&0\quad\mbox{on}\quad\partial\Omega.
	\end{array}
	\right. \tag{$P_{1,\alpha}^t$}\label{P1a}
	\end{equation}

\begin{theorem} \label{LP}
	Assume that $K_1,K_2\in \mathcal{K}$ and the conditions 
	 $(f_0)-(f_2)$ hold.    If $\operatorname{div}\vec{\alpha}(x)=0$ in $\Omega$ then problem  \eqref{P1a}	has a positive solution if, and only if, $t>t_1^\alpha:=\lambda_1^{\alpha}/\lambda_A$.
\end{theorem}
In order to prove this result, firstly we need some auxiliary lemmas. The following result provides a priori bounds for the positive solutions of \( P \).

\begin{lemma}\label{L1}
	Let $U=(u,v)$ be a positive solution to \eqref{P1a}, with $t\leq\Lambda$. Then,
	\begin{equation*}
		\|U\|_1\leq M_{\Lambda},
	\end{equation*}
	where $M_{\Lambda}$ is a positive constant.
\end{lemma}
\begin{proof}
Set the norm in $H=H_{0}^{1}(\Omega)\times H_{0}^{1}(\Omega)$ as given by
$$\|U\|_H=\|u\|_{H_{0}^{1}(\Omega)}+\|v\|_{H_{0}^{1}(\Omega)} \ \mbox{for all  } U=(u,v) \in H.$$
Suppose, by contradiction, that there exists $(u_n,v_n)\subset H$ with $(u_n,v_n)>0$ and $(t_n)\subset[0,\Lambda]$ such that $\|U_n\|_H\rightarrow\infty$, where $U_n=(u_n,v_n)$. So
\begin{equation*}
\int_{\Omega}\nabla u_n\nabla\varphi dx+\int_{\Omega}(\vec{\alpha}(x)\cdot \nabla u_n)\varphi dx+\int_{\Omega}\Phi_{(u_n,v_n)}u_n\varphi dx=t_n\int_{\Omega}(au_n+bv_n)\varphi dx
\end{equation*}
and
\begin{equation*}
	\int_{\Omega}\nabla v_n\nabla\eta dx+\int_{\Omega}(\vec{\alpha}(x)\cdot \nabla v_n)\eta dx+\int_{\Omega}\Psi_{(u_n,v_n)}v_n\eta dx=t_n\int_{\Omega}(cu_n+dv_n)\eta dx
\end{equation*}
for all $\varphi,\eta\in H_0^1(\Omega)$. Thus, setting $\overline{u}_n=u_n/\|U_n\|$ e $\overline{v}_n=v_n/\|U_n\|$, we get
\begin{equation*}
	\int_{\Omega}\nabla \overline{u}_n\nabla\varphi dx+\int_{\Omega}(\vec{\alpha}(x)\cdot \nabla \overline{u}_n)\varphi dx+\int_{\Omega}\Phi_{(u_n,v_n)}\overline{u}_n\varphi dx=t_n\int_{\Omega}(a\overline{u}_n+b\overline{v}_n)\varphi dx
\end{equation*}
and
\begin{equation*}
	\int_{\Omega}\nabla \overline{v}_n\nabla\eta dx+\int_{\Omega}(\vec{\alpha}(x)\cdot \nabla \overline{v}_n)\eta dx+\int_{\Omega}\Psi_{(u_n,v_n)}\overline{v}_n\eta dx=t_n\int_{\Omega}(c\overline{u}_n+d\overline{v}_n)\eta dx.
\end{equation*}

Once $(\overline{u}_n)$ and $(\overline{v}_n)$ are bounded in $H_0^1(\Omega)$, along some subsequence, relabeled by $n$, 
\begin{eqnarray*}
	\overline{u}_n\rightharpoonup u\mbox{ in }H_{0}^{1}(\Omega),\overline{u}_n\rightarrow u\mbox{ in }L^2(\Omega)\mbox{ and }\overline{u}_n(x)\rightarrow u(x)\mbox{ a.e. in }\Omega\\
	\overline{v}_n\rightharpoonup v\mbox{ in }H_{0}^{1}(\Omega),\overline{v}_n\rightarrow v\mbox{ in }L^2(\Omega)\mbox{ and }\overline{v}_n(x)\rightarrow v(x)\mbox{ a.e. in }\Omega.
\end{eqnarray*}
Taking $\varphi=\overline{u}_n/\|U_n\|_H^{\gamma}$ and $\eta=\overline{v}_n/\|U_n\|_H^{\gamma}$ as test functions, and recalling that $t^{\gamma}\Phi_{(u_n,v_n)}=\Phi_{(tu_n,tv_n)}$ and  $t^{\gamma}\Psi_{(u_n,v_n)}=\Psi_{(tu_n,tv_n)}$, for all $t>0$, we derive
\begin{eqnarray*}
	\frac{1}{\|U_n\|_{H}^{\gamma}}\|\overline{u}_n\|^{2}_{H_{0}^{1}(\Omega)}+\frac{1}{\|U_n\|_{H}^{\gamma}}\int_{\Omega}(\vec{\alpha}(x)\cdot \nabla \overline{u}_n) \overline{u}_n dx+\int_{\Omega}\Phi_{(\overline{u}_n,\overline{v}_n)}\overline{u}_n^2dx=t_n\int_{\Omega}(a\overline{u}_n+b\overline{v}_n)\frac{\overline{u}_n}{\|U_n\|_{H}^{\gamma}}dx\\
	\frac{1}{\|U_n\|_{H}^{\gamma}}\|\overline{v}_n\|^{2}_{H_{0}^{1}(\Omega)}+\frac{1}{\|U_n\|_{H}^{\gamma}}\int_{\Omega}(\vec{\alpha}(x)\cdot \nabla \overline{v}_n) \overline{v}_n dx+\int_{\Omega}\Psi_{(\overline{u}_n,\overline{v}_n)}\overline{v}_n^2dx=t_n\int_{\Omega}(c\overline{u}_n+d\overline{v}_n)\frac{\overline{v}_n}{\|U_n\|_{H}^{\gamma}}dx.
\end{eqnarray*}
Therefore, using Hölder's inequality and the inequalities above, we get
\begin{equation*}
	\lim_{n\rightarrow\infty}\int_{\Omega}\Phi_{(\overline{u}_n,\overline{v}_n)}\overline{u}_n^2dx=\lim_{n\rightarrow\infty}\int_{\Omega}\Psi_{(\overline{u}_n,\overline{v}_n)}\overline{v}_n^2dx=0.
\end{equation*}
By Fatou's lemma,
\begin{eqnarray*}
0\leq\int_{\Omega}\Phi_{(u,v)}u^2dx\leq\lim_{n\rightarrow\infty}\int_{\Omega}\Phi_{(\overline{u}_n,\overline{v}_n)}\overline{u}_n^2dx=0 \mbox{ and }
0\leq\int_{\Omega}\Psi_{(u,v)}v^2dx\leq\lim_{n\rightarrow\infty}\int_{\Omega}\Psi_{(\overline{u},\overline{v}_n)}\overline{v}_n^2dx=0.
\end{eqnarray*}
Hence,
\begin{equation*}
\int_{\Omega\times\Omega}K_1(x,y)f(|u(y)|,|v(y)|)|u(x)|^2dxdy=\int_{\Omega\times\Omega}K_2(x,y)g(|u(y)|,|v(y)|)|v(x)|^2dxdy=0.
\end{equation*}
Thus, by $(f_1)$,
\begin{eqnarray*}
0\leq\varepsilon_0\int_{\Omega\times\Omega}K_1(x,y)|u(y)|^{\gamma}|u(x)|^2dxdy\leq\int_{\Omega\times\Omega}K_1(x,y)f(|u(y)|,|v(y)|)|u(x)|^2dxdy=0,\\
0\leq\varepsilon_0\int_{\Omega\times\Omega}K_2(x,y)|v(y)|^{\gamma}|v(x)|^2dxdy\leq\int_{\Omega\times\Omega}K_2(x,y)g(|u(y)|,|v(y)|)|v(x)|^2dxdy=0.
\end{eqnarray*}
	 Consequently,
\begin{equation*}
\int_{\Omega\times\Omega}K_1(x,y)|u(y)|^{\gamma}|u(x)|^2dxdy=\int_{\Omega\times\Omega}K_2(x,y)|v(y)|^{\gamma}|v(x)|^2dxdy=0.
\end{equation*}
Since $K_1,K_2\in\mathcal{K}$, using Remark \ref{remar1}, we get that $u=v=0$. Thus, $(\overline{u}_n)$ and $(\overline{v}_n)$ converge to 0 in $L^{2}(\Omega)$. Considering $\varphi=\overline{u}_n$ and $\eta=\overline{v}_n$ as test functions, we get that
\begin{eqnarray*}
\int_{\Omega}|\nabla\overline{u}_n|^2dx+\int_{\Omega}(\vec{\alpha}(x)\cdot \nabla \overline{u}_n)\overline{u}_n dx+\int_{\Omega}\Phi_{(u_n,v_n)}\overline{u}_n^2dx=t_n\int_{\Omega}(a\overline{u}_n+b\overline{v}_n)\overline{u}_ndx \mbox{ and }\\
\int_{\Omega}|\nabla\overline{v}_n|^2dx+\int_{\Omega}(\vec{\alpha}(x)\cdot \nabla \overline{v}_n)\overline{v}_n dx+\int_{\Omega}\Psi_{(u_n,v_n)}\overline{v}_n^2dx=t_n\int_{\Omega}(c\overline{u}_n+d\overline{v}_n)\overline{v}_ndx.
\end{eqnarray*}	
As $(t_n)$ is bounded by $\Lambda$, we get
\begin{eqnarray*}
\int_{\Omega}|\nabla\overline{u}_n|^2dx\leq \Lambda\left[a\int_{\Omega}|\overline{u}_n|^2dx+b\int_{\Omega}|\overline{u}_n\overline{v}_n|dx\right]\mbox{ and }
\int_{\Omega}|\nabla\overline{v}_n|^2dx\leq \Lambda\left[c\int_{\Omega}|\overline{u}_n\overline{v}_n|dx+d\int_{\Omega}|\overline{v}_n|^2dx\right].
\end{eqnarray*}
Consequently, $\|(\overline{u}_n,\overline{v}_n)\|_H\rightarrow0$, which contradicts the fact that $\|(\overline{u}_n,\overline{v}_n)\|_H=1$ for all $n\in\N$. Thus, $(u_n,v_n)$ is bounded in $H$. By elliptic regularity, we obtain that $(u_n,v_n)$ is bounded in $E_1$, completing the proof.
\end{proof}

Now, we present a nonexistence result.

\begin{lemma}\label{lem:nosolution}
	The problem \eqref{P1a} does not admit a positive solution if $t\leq t_1^\alpha= \lambda_1^{\alpha}/\lambda_A$.
\end{lemma}

\begin{proof}
If $(t,U)$ is a positive solution to \eqref{P1a},  we have
$$
\left\{
\begin{array}{lcl}
	-\Delta u+\vec{\alpha}(x)\cdot \nabla u+\Phi_{(u,v)}u&=&t\left[au+\frac{b}{\sigma}(\sigma v)\right]\quad\mbox{in}\quad\Omega,\\
	-\Delta (\sigma v)+\vec{\alpha}(x)\cdot \nabla (\sigma v)+\Psi_{(u,v)}(\sigma v)&=&t[(c\sigma)u+d(\sigma v)]\quad\mbox{in}\quad\Omega,\\
	\qquad \qquad \qquad \qquad \qquad \qquad	u, v &>&0\quad\mbox{in}\quad\Omega,\\
\qquad \qquad \qquad \qquad \qquad \qquad	u=v&=&0\quad\mbox{on}\quad\partial\Omega,
\end{array}
\right.
$$
for all $\sigma>0$. In particular, for $\sigma^2=b/c$, we may fix $w=\sigma v$ and observe that $\hat{b}:=b/\sigma=c\sigma$. Hence, $0< U_0=(u,w)\in E$ is solution to the problem 
\begin{equation}\label{u0}
\left\{
\begin{array}{lcl}
	-\Delta u+\vec{\alpha}(x)\cdot \nabla u+\Phi_{(u,v)}u&=&t[au+\hat{b}w],\quad\mbox{in}\quad\Omega\\
	-\Delta w+\vec{\alpha}(x)\cdot \nabla w+\Psi_{(u,v)}w&=&t[\hat{b}u+dw],\quad\mbox{in}\quad\Omega\\
\quad \qquad \qquad \qquad \qquad	u=w&=&0,\quad\mbox{on}\quad\partial\Omega.
\end{array}
\right.
\end{equation}
Clearly, $$A_0=\left(\begin{array}{cc}a & \hat{b}\\ \hat{b} & d\end{array}\right)$$
is a symmetric matrix, whose largest eigenvalue is $\lambda_A>0$, as well in the matrix $A$. Now, we consider $z=(z_1,z_2)>0$, an eigenvector of $A_0$ associated to $\lambda_A$  and $\varphi_1^{\alpha}>0$ an eigenfunction of $-\Delta+\vec{\alpha}\cdot\nabla$ associated to $\lambda_1^{\alpha}$. By a direct calculation we find that
$$
L(X)=t_1^{\alpha} A_0X,
$$
where $X:=\varphi_1^{\alpha}z$. Now we define
$$
\mathfrak{P}_{U_0}:= (\Phi_{(u,v)}u,\Psi_{(u,v)}w). 
$$
With this notation, the system (\ref{u0}) can be rewritten as:
$$
L(U_0)=tA_0U_0-\mathfrak{P}_{U_0}.
$$
Calculating the inner product with \( X \) and integrating over \( \Omega \), we obtain 
\begin{equation*}
\int_{\Omega}\langle L(U_0),X\rangle dx=t\int_{\Omega}\langle A_0U_0,X\rangle dx-\int_{\Omega}\langle\mathfrak{P}_{U_0},X \rangle<t\int_{\Omega}\langle A_0U_0,X\rangle dx.
\end{equation*}
Since $\mathrm{div}\,\vec{\alpha} = 0$, $L$ is self-adjoint. Consequently,
\begin{equation*}
\int_{\Omega}\langle L(U_0),X\rangle dx=	\int_{\Omega}\langle U_0,L(X)\rangle dx=t_1^{\alpha}\int_{\Omega}\langle U_0,A_0X\rangle dx=t_1^{\alpha}\int_{\Omega}\langle A_0U_0,X\rangle dx,
\end{equation*}
and so
\begin{equation*}
t_1^{\alpha}\int_{\Omega}\langle A_0U_0,X\rangle dx<t\int_{\Omega}\langle A_0U_0,X\rangle dx.
\end{equation*}
Furthermore, observing that $\int_{\Omega}\langle A_0U_0,X\rangle dx>0$, we have $t>t_1^{\alpha}$, concluding the proof.
\end{proof}

Now, we will apply the global bifurcation theorems to obtain an unbounded continuum of positive solutions to \eqref{P1a}. To this end, we first observe that, by elliptic regularity, \( U \in E_1 \) is a classical positive solution of \eqref{P1a} if and only if \( \mathfrak{F}(t,U) = 0 \), where \( \mathfrak{F}:\mathbb{R} \times E_1 \to E_1 \) is given by:

$$
\mathfrak{F}(t,U) = U-tS_1(U) -G_1(U),
$$
where $S_1,G_1:E_1 \to E_1$ are the operators defined in Section 2. Moreover, we denote
 $$
\Sigma=\overline{\{(t,U)\in\R\times E_1:\mathfrak{F}(t,U)=0, \,U\neq0\}}.
$$

With these considerations, we obtain the following result.

\begin{proposition}
There exists a component $\mathcal{C}=\mathcal{C}_{t_1^{\alpha}}\subset\Sigma$ of solutions to problem \eqref{P1a} emanating from the trivial solution at $(t_1^\alpha,0)$ which satisfies one of the following non-excluding options: either (i) $\mathcal{C}$ is unbounded $\mathbb{R} \times E_1$ or (ii) there exists $\hat{t} \neq t_1^\alpha$ such that $(\hat{t},0)\in \mathcal{C}$.
\end{proposition}
\begin{proof}
    This result follows from the classical global bifurcation results of Rabinowitz \cite{Rabinowitz}. Indeed, under assumptions \((f_2)\) and \((\Phi_2)\), we have that \( G_1(U) = o(\|U\|_1) \) as \( U \to 0 \) (see Section 2). Moreover, note that \( t \) is a characteristic value of \( S_1 \) if and only if \( t \) is an eigenvalue of \( L \). Since \( t = t_1^\alpha \) is an eigenvalue of \( L \) with algebraic multiplicity $1$, the result follows from \cite[Theorem 1.3]{Rabinowitz}.
\end{proof}

In what follows, we will show that \( \mathcal{C} \) has an unbounded subcontinnum of positive solutions to \eqref{P1a}. To this end, we will need the following auxiliary results.

\begin{lemma}\label{sinal}
	There exists $\delta>0$ such that, if $(t,U)\in\mathcal{C}$ with $|t-t_1^{\alpha}|+\|U\|_1<\delta$ and $U\neq0$, then $U$ has a defined sign, that is, either
	$$
	U(x)>0, \quad \forall x\in\Omega\quad \mbox{or} \quad U(x)<0,\quad \forall x\in\Omega.
	$$
\end{lemma}
\begin{proof}
	It is enough to prove that for any pair of sequences $(U_n) \subset E_1$ and $t_{n}\to t_1^{\alpha}$ with 
	$$
	U_n\neq0,\quad \|U_n\|_1\to 0 \quad \mbox{and } \quad U_{n}=t_{n}S_1(U_n)+G_1(U_n),
	$$
	$U_n$ does not change sign for sufficiently large $n$.
	Setting $W_n=U_n/\|U_n\|_1$, we have that
	$$W_n=t_nS_1(W_n)+\frac{G_1(U_n)}{\|U_n\|_1}=t_nS_1(W_n)+o_n(1).$$
	From the compactness of the operator $S_1$, up to subsequence, we can assume that $(S_1(W_n))$ is convergent. Then, $W_n\to W$ in $E_1$ for some $W \in E_1$, with $\|W\|_1=1$. Consequently, 
	$$
	\left\{\begin{array}{lcl}
		LW&=&t_1^{\alpha}AW\mbox{ in }\ \Omega,\\
		W&=&0\mbox{ on }\ \partial\Omega.
	\end{array}\right.
	$$	
	Since $W\neq0$, by \cite[Lemmas 3.1 and 3.2]{RM}, we conclude that
	$W(x)>0\mbox{ or }W(x)<0,\mbox{ for all }x\in\Omega.$
	Therefore, without loss of generality, we may assume $W>0$ in $\Omega$, consequently, $W_n>0$ in $\Omega$ for $n$ large enough. Since  $U_n$ and $W_n$ have the same sign, it follows that $U_n$ is also positive, completing the proof.
\end{proof}

Now, we will provide a more precise description of the solutions that belong to \(\mathcal{C}\).

\begin{lemma} It holds that
$$\mathcal{C}= \mathcal{C}^+ \cap \mathcal{C}^-,$$
where 
$
    \mathcal{C}^+:=\{(t,U)\in \mathcal{C};\, U(x)>0 ~~\forall x \in \Omega \} \cup \{(t_1^\alpha,0)\}$ 
    and $
    \mathcal{C}^-:=\{(t,U)\in \mathcal{C};\, (t,-U)\in \mathcal{C}^+\}$.    
\end{lemma}
\begin{proof}
First we will prove that   $\mathcal{C}$ does not possess semi-trivial positive solutions $(t,u,0)$ or $(t,0,v)$. Indeed, suppose $(t,u,0)\in\mathcal{C}$, with $u\geq 0$ and $u \neq0$. From the arguments of Lemma \ref{lem:nosolution}, we are already assuming that \( t > t_1^\alpha \), and so by the system \eqref{P1a}, it follows that
$$-\Delta u+\vec{\alpha}(x)\cdot\nabla u+\Phi_{(u,0)}(x)u = t au\geq0\quad\mbox{ in }\Omega,$$
Then, applying again the maximum principle, see \cite{Gilbarg-Trudinger} or \cite{LopezGomes}, we derive $u(x)>0$ for all $x \in \Omega$. Using the second equation of the system \eqref{P1a}, we obtain $tcu=0$, which is a contradiction. Similarly, $(t,0,v)\notin\mathcal{C}$.

On the other hand, 	since $\mathcal{C}$	 and considering the previous Lemma, we can argue as in \cite[Remark 4.1]{RM} to ensure that there does not exist \( (t, U) = (t, u, v) \in \mathcal{C} \) such that \( u < 0 \) and \( v > 0 \) or \( u > 0 \) and \( v < 0 \). Moreover, by the strong maximum principle, if $(t,U)$ is a nonnegative solution of \eqref{P1a}, then $U(x)>0$ for all $x \in \Omega$.

Finally, we observe that $(t,U)\in \mathcal{C}$ if, and only if $(t,-U)\in \mathcal{C}$. Consequently, we obtain the decomposition $\mathcal{C}=\mathcal{C}^+\cup \mathcal{C}^-$.
\end{proof}

%\begin{remark}\label{remar3}
% Furthermore,  $\mathcal{C}$ does not possess nonnegative semi-trivial solutions $(t,u,0)$ or $(t,0,v)$. 
%\end{remark}

%\begin{remark}\label{remar4}
%	In order to obtain the decomposition of $\mathcal{C}$, it was essential that $a,b,c,d>0$. If such a condition is not satisfied, such an argument cannot be used.
%\end{remark}

%\begin{remark}\label{remark5}
%    By the previous lemma and Remark \ref{remar3}, all solutions in \( \mathcal{C} \) have a defined sign. 
%    Moreover,   Thus, we can get the decomposition \( \mathcal{C} = \mathcal{C}^+ \cup \mathcal{C}^- \), where
%    $
%    \mathcal{C}^+:=\{(t,U)\in \mathcal{C};\, U(x)>0 ~~\forall x \in \Omega \} \cup \{(t_1^\alpha,0)\}$ 
%    and $
%    \mathcal{C}^-:=\{(t,U)\in \mathcal{C};\, (t,-U)\in \mathcal{C}^+\}$.
%\end{remark}

We are now in a position to establish the following result:

\begin{lemma} \label{CUB}
	$\mathcal{C}^{+}$ is unbounded.
\end{lemma} 

\begin{proof}
	Suppose, by contradiction, that $\mathcal{C}^{+}$ is bounded. Then, $\mathcal{C}$ is also bounded. From the global bifurcation theorem, there exists $(\hat{t},0)\in\mathcal{C}$, where $\hat{t}\neq t_1^{\alpha}$ and $\hat{t}^{-1}\in\sigma(S_1)$.
Hence, without loss of generality, there exists $(t_n,U_n)\subset\mathcal{C}^+$ with $t_n\rightarrow\hat{t}$ such that
$$U_n\neq0,\quad\|U_n\|_1\rightarrow0\mbox{ and }U_n=t_{n}S_1(U_n)+G_1(U_n).$$	
Setting $W_n=U_n/\|U_n\|_1$, from the compactness of the operator $S_1$,  we can assume that there exists $W\in E_1$  such that, up to subsequence, $W_n\rightarrow W$ in $E_1$, where $W\neq0$, $W\geq0$, and satisfies
$$
\left\{\begin{array}{lcl}
LW&=&(\hat{t}A)W\mbox{ in }\quad\Omega,\\
W&=&0\mbox{ in }\quad\partial\Omega.
\end{array}\right.
$$
Therefore $\hat{t}=t_1^{\alpha}$, which is a contradiction. This proves the lemma.
\end{proof}

\begin{proof}[Proof of Theorem \ref{LP}]
	In view of Lemma \ref{lem:nosolution}, there is no positive solution to problem \eqref{P1a}, when $t\leq t_1^\alpha$. Moreover, Lemma \ref{L1} provides a priori bounds for the positive solutions of $\eqref{P1a}$. Consequently, the unbounded continuum $\mathcal{C}^+$ intersects ${t} \times E_1$ for all \( t > t_1^\alpha \). This completes the proof.
\end{proof}

\begin{proof}[Proof of Theorem \ref{TP1}]
In view of Theorem \ref{LP}, it is clear that the problem \eqref{P} has a positive solution if, and only if, $1> t_1^{\alpha}=\lambda_1^{\alpha}/\lambda_A$. Therefore, \eqref{P} has a positive solution if, and only if, $\lambda_A>\lambda_1^{\alpha}$. 
\end{proof}
\section{Existence and nonexistence of solution to $p,q>1$}

Similarly to what was done in the previous section, we will consider the following auxiliary system:  
\begin{equation}\tag{$P_{p,q}^t$}\label{Ppqt}
    	\left\{
	\begin{array}{lcl}
		-\Delta u+p|u|^{p-1}\vec{\alpha}(x)\cdot \nabla u+\Phi_{(u,v)}u&=&t(au+bv)\quad\mbox{in}\quad\Omega,\\
		-\Delta v+q|v|^{q-1}\vec{\beta}(x)\cdot \nabla v+\Psi_{(u,v)}v&=&t(cu+dv)\quad\mbox{in}\quad\Omega,\\
\qquad \qquad 	\qquad \qquad \qquad \qquad	u=v&=&0\quad\mbox{on}\quad\partial\Omega,
	\end{array}
	\right.
\end{equation}
where $t \in \mathbb{R}$ will be regard as bifurcation parameter. Our main result on the existence and nonexistence of positive solutions for this system can be stated as follows:

\begin{theorem} \label{LP2}
	Assume that $K_1,K_2\in \mathcal{K}$ and $(f_0)-(f_2)$ hold with $p,q>1$ and $\gamma>\max\{p,q\}$. If $\operatorname{div}\vec{\alpha}(x)=\operatorname{div}\vec{\beta}(x)=0$ in $\Omega$, then problem \eqref{Ppqt} has a positive solution if, and only if, $t>t_1:=\lambda_1/\lambda_A$.
    \end{theorem}

To establish this result, we need to prove non-existence results and a priori bounds, which will be done in the following lemmas.

\begin{lemma}\label{L1.}
	Let $U=(u,v)$ be a positive solution to $(P^{t}_{p,q})$ with $t\leq\Lambda$, then there exists a constant $M_{\Lambda}>0$ such that
	\begin{equation*}
		\|U\|_1\leq M_{\Lambda}.
	\end{equation*}

\end{lemma}
\begin{proof}
Suppose, by contradiction, that $(u_n,v_n)\subset H$ with $(u_n,v_n)>0$ and $(t_n)\subset[0,\Lambda]$ such that $\|U_n\|_H\rightarrow\infty$, where $U_n=(u_n,v_n)$. So
\begin{equation*}
	\int_{\Omega}\nabla u_n\nabla\varphi dx+\int_{\Omega}p|u_n|^{p-1}(\vec{\alpha}(x)\cdot \nabla u_n)\varphi dx+\int_{\Omega}\Phi_{(u_n,v_n)}u_n\varphi dx=t_n\int_{\Omega}(au_n+bv_n)\varphi dx
\end{equation*}
and
\begin{equation*}
	\int_{\Omega}\nabla v_n\nabla\eta dx+\int_{\Omega}q|v_n|^{q-1}(\vec{\beta}(x)\cdot \nabla v_n)\eta dx+\int_{\Omega}\Psi_{(u_n,v_n)}v_n\eta dx=t_n\int_{\Omega}(cu_n+dv_n)\eta dx
\end{equation*}
for all $\varphi,\eta\in H_0^1(\Omega)$. Thus, setting $\overline{u}_n=u_n/\|U_n\|$ e $\overline{v}_n=v_n/\|U_n\|$, we get
\begin{equation*}
	\int_{\Omega}\nabla \overline{u}_n\nabla\varphi dx+\int_{\Omega}p|u_n|^{p-1}(\vec{\alpha}(x)\cdot \nabla \overline{u}_n)\varphi dx+\int_{\Omega}\Phi_{(u_n,v_n)}\overline{u}_n\varphi dx=t_n\int_{\Omega}(a\overline{u}_n+b\overline{v}_n)\varphi dx
\end{equation*}
and
\begin{equation*}
	\int_{\Omega}\nabla \overline{v}_n\nabla\eta dx+\int_{\Omega}q|v_n|^{q-1}(\vec{\beta}(x)\cdot \nabla \overline{v}_n)\eta dx+\int_{\Omega}\Psi_{(u_n,v_n)}\overline{v}_n\eta dx=t_n\int_{\Omega}(c\overline{u}_n+d\overline{v}_n)\eta dx.
\end{equation*}
Once $(\overline{u}_n)$ and $(\overline{v}_n)$ are bounded in $H_0^1(\Omega)$, we can suppose that there is $u$ and $v$ in $H_0^1(\Omega)$ such that,  up to subsequence,
\begin{eqnarray*}
	\overline{u}_n\rightharpoonup u, 	\overline{v}_n\rightharpoonup v \mbox{ in }H_{0}^{1}(\Omega),\overline{u}_n\rightarrow u, \overline{v}_n\rightarrow v\mbox{ in }L^2(\Omega)\mbox{ and }\overline{u}_n(x)\rightarrow u(x), \overline{v}_n(x)\rightarrow v(x)\mbox{ a.e. in }\Omega.
\end{eqnarray*}
Taking $\varphi=\overline{u}_n/\|U_n\|_H^{\gamma}$ and $\eta=\overline{v}_n/\|U_n\|_H^{\gamma}$ as test functions and recalling that $t^{\gamma}\Phi_{(u_n,v_n)}=\Phi_{(tu_n,tv_n)}$ and  $t^{\gamma}\Psi_{(u_n,v_n)}=\Psi_{(tu_n,tv_n)}$, for all $t>0$, we get
\begin{eqnarray*}
	\frac{1}{\|U_n\|_{H}^{\gamma}}\|\overline{u}_n\|^{2}_{H_{0}^{1}(\Omega)}+\frac{1}{\|U_n\|_{H}^{\gamma-p}}\int_{\Omega}p|\overline{u}_n|^{p-1}(\vec{\alpha}(x)\cdot \nabla \overline{u}_n) \overline{u}_n dx+\int_{\Omega}\Phi_{(\overline{u}_n,\overline{v}_n)}\overline{u}_n^2dx=t_n\int_{\Omega}(a\overline{u}_n+b\overline{v}_n)\frac{\overline{u}_n}{\|U_n\|_{H}^{\gamma}}dx,\\
	\frac{1}{\|U_n\|_{H}^{\gamma}}\|\overline{v}_n\|^{2}_{H_{0}^{1}(\Omega)}+\frac{1}{\|U_n\|_{H}^{\gamma-q}}\int_{\Omega}q|\overline{v}_n|^{q-1}(\vec{\beta}(x)\cdot \nabla \overline{v}_n) \overline{v}_n dx+\int_{\Omega}\Psi_{(\overline{u}_n,\overline{v}_n)}\overline{v}_n^2dx=t_n\int_{\Omega}(c\overline{u}_n+d\overline{v}_n)\frac{\overline{v}_n}{\|U_n\|_{H}^{\gamma}}dx.
\end{eqnarray*}
Therefore, using the Hölder's inequality and the equalities above, we get
\begin{equation*}
\lim_{n\rightarrow\infty}\int_{\Omega}\Phi_{(\overline{u}_n,\overline{v}_n)}\overline{u}_n^2dx=\lim_{n\rightarrow\infty}\int_{\Omega}\Psi_{(\overline{u}_n,\overline{v}_n)}\overline{v}_n^2dx=0.
\end{equation*}
Consequently, repeating the argument used in Lemma \ref{L1}, it yields $\|(\overline{u}_n,\overline{v}_n)\|_H\rightarrow0$. This contradicts the fact that $\|(\overline{u}_n,\overline{v}_n)\|_H=1$ for all $n\in\N$. So, since $(\overline{u}_n,\overline{v}_n)$ is bounded in $H$, by standard arguments we conclude that $(\overline{u}_n,\overline{v}_n)$ is bounded in $E_1$, completing the proof.
\end{proof}

\begin{lemma}\label{lem:nosolution2}
The problem $(P^{t}_{p,q})$ does not admit a positive solution if $t\leq t_1=\lambda_1  /\lambda_A$.
\end{lemma}

\begin{proof}
	Suppose that $(u,v)$ is a positive solution to \eqref{Ppqt} for some $t \in \mathbb{R}$ and define \( U_0 := (u, w) \), where \( w := \sigma v \) and \( \sigma := b/c \). By reasoning as in the proof of Lemma \ref{lem:nosolution}, we obtain that \( U_0 \) is a solution of the following system  
	$$
	\left\{
	\begin{array}{lcl}
		-\Delta u+pu^{p-1}\vec{\alpha}(x)\cdot \nabla u+\Phi_{(u,v)}u&=&t[au+\hat{b}w]\quad\mbox{in}\quad\Omega,\\
		-\Delta w+\sigma^{1-q}qw^{q-1}\vec{\beta}(x)\cdot \nabla w+\Psi_{(u,v)}w&=&t[\hat{b}u+dw]\quad\mbox{in}\quad\Omega,\\
\qquad \qquad \qquad \qquad \qquad \qquad		u=w&=&0\quad\mbox{on}\quad\partial\Omega,
	\end{array}
	\right.
	$$
where $\hat{b}=b/\sigma= c \sigma$. On the other hand, since
$$
A_0=\left(\begin{array}{cc}a & \hat{b}\\ \hat{b} & d\end{array}\right)
$$
is a symmetric matrix, we know that  
	\begin{equation*}
		\left<A_0z,z\right>\leq \lambda_A|z|^2,\quad \forall z\in\R^2,
	\end{equation*}
Using this inequality with $z=U_0(x)$ and integrating over $\Omega$ it becomes apparent 
\begin{eqnarray*}
 t \lambda_A \int_{\Omega} (|u|^2+|w|^2) dx&\geq& \int_{\Omega}\left<tA_0U_0, U_0\right>dx \\
 &=& \int_{\Omega}|\nabla u|^2+pu^{p-1}\vec{\alpha}(x)\cdot \nabla u+\Phi_{(u,v)}(x)u^2dx \\
 && + \int_{\Omega}|\nabla w|^2+q\sigma^{1-q}w^{q-1} \vec{\beta}(x)\cdot \nabla w+\Psi_{(u,v)}(x)w^2dx\\ 
 &>& \int_{\Omega}(|\nabla u|^2+|\nabla v|^2)dx\\
 &\geq& \lambda_1\int_{\Omega}(|u|^2+|v|^2)dx,
\end{eqnarray*}
where the Poincaré inequality and Remark \ref{remar2} were used. Therefore, we obtain that $t > t_1=\lambda_1/\lambda_A$.
\end{proof}

To apply the global bifurcation results, we define the operator \(\mathfrak{F}_{pq}: \mathbb{R} \times E_1 \to E_1\) given by  
$$  
\mathfrak{F}_{pq}(t, U) =  U-tS_{pq}(U) -G_{pq}(U),  
$$  
and consider the set  
$$  
\Sigma_{pq}=\overline{\{(t,U)\in\R\times E_1:\mathfrak{F}_{pq}(t, U)=0,U\neq0\}},
$$  
Using the same arguments as in the previous section, we obtain the following result.

\begin{proposition}
There exists a component $\mathcal{C}=\mathcal{C}_{t_1}\subset\Sigma_{pq}$ of solutions to problem \eqref{Ppqt} emanating from the trivial solution at $(t_1,0)$ which satisfies one of the following non-excluding options: either (i) $\mathcal{C}$ is unbounded $\mathbb{R} \times E_1$ or (ii) there exists $\hat{t} \neq t_1$ such that $(\hat{t},0)\in \mathcal{C}$.
\end{proposition}

In order to establish that \(\mathcal{C}\) contains an unbounded subcontinuum of positive solutions of \eqref{Ppqt}, we require the following auxiliary result.

\begin{lemma}\label{sinal1}
	There exists $\delta>0$ such that, if $(t,U)\in\mathcal{C}$ with $|t-t_1|+\|U\|_1<\delta$ and $U\neq0$, then $U$ has a defined sign, that is,
	$$
	U(x)>0, \quad \forall x\in\Omega\quad \mbox{or} \quad U(x)<0,\quad \forall x\in\Omega.
	$$
\end{lemma}
\begin{proof}
It is enough to prove that for any sequences $(U_n) \subset E_1$ and $t_{n}\to t_1$ with 
$$
U_n\neq0,\quad \|U_n\|_1\to 0 \quad \mbox{and } \quad U_{n}=t_{n}S_{pq}(U_n)+G_{pq}(U_n),
$$
$U_n$ does not change sign for sufficiently large $n$.
Setting $W_n=U_n/\|U_n\|_1$, we have that
$$W_n=t_nS_{pq}(W_n)+\frac{G_{pq}(U_n)}{\|U_n\|_1}=t_nS_{pq}(W_n)+o_n(1).$$
From the compactness of the operator $S_{pq}$, we can assume that $(S_{pq}(W_n))$ is convergent. Then, $W_n\to W$ in $E_1$ for some $W \in E_1$, with $\|W\|_1=1$. Consequently, 
$$
\left\{\begin{array}{lcl}
-\Delta W&=&t_1AW\quad \mbox{in}\quad\Omega,\\
W&=&0\quad\mbox{on}\quad\partial\Omega.
\end{array}\right.
$$	
The rest of the proof follows the same lines as the proof of Lemma \ref{sinal}.
%Since $W\neq0$,  by  \cite[Lemmas 3.1 and 3.2]{RM}, it follows that$W(x)>0\mbox{ or }W(x)<0,\mbox{ for all }x\in\Omega.$Therefore, without loss of generality, we  may assume $W>0$ in $\Omega$ and, consequently, $W_n>0$ in $\Omega$ for $n$ large enough. Since  $U_n$ and $W_n$ has the same sign, we have that $U_n$ is also positive, completing the proof.
\end{proof}

Using the arguments from the previous section, we can once again obtain the following decomposition of \(\mathcal{C}\):

\begin{lemma} It holds that
$$\mathcal{C}= \mathcal{C}^+ \cap \mathcal{C}^-,$$
where 
$
    \mathcal{C}^+:=\{(t,U)\in \mathcal{C};\, U(x)>0 ~~\forall x \in \Omega \} \cup \{(t_1^\alpha,0)\}$ 
    and $
    \mathcal{C}^-:=\{(t,U)\in \mathcal{C};\, (t,-U)\in \mathcal{C}^+\}$.     Moreover, $\mathcal{C}^+$ is unbounded.
\end{lemma}

\begin{comment}
\begin{lemma} \label{CUB1}
	$\mathcal{C}^{+}$ is unbounded.
\end{lemma} 

\begin{proof}
Suppose, by contradiction, that $\mathcal{C}^{+}$ is bounded. Then, $\mathcal{C}$ is also bounded. From the global bifurcation theorem, there exists $(\hat{t},0)\in\mathcal{C}$ , where $\hat{t}\neq t_1$ and $\hat{t}^{-1}\in\sigma(S_{pq})$.
	
Hence, without loss of generality, there exists $(t_n,U_n)\subset\mathcal{C}^+$ with $t_n\rightarrow\hat{t}$ such that
$$U_n\neq0,\quad\|U_n\|_1\rightarrow0\mbox{ and }U_n=F_{pq}(t_n,U_n).$$
Setting $W_n=U_n/\|U_n\|_1$, similarly to what was done in the previous lemma, there exists $W\in E_1$ with $W_n\rightarrow W$ in $E_1$, where $W\neq0$, $W\geq0$ and satisfies
$$
\left\{\begin{array}{lcl}
	-\Delta W&=&\hat tAW\quad \mbox{in}\quad\Omega,\\
	W&=&0\quad\mbox{on}\quad\partial\Omega.
\end{array}\right.
$$
Therefore, $\hat{t}\lambda=\lambda_1$ and, consequently, $\hat{t}=t_1$, which is impossible. This proves the lemma.
\end{proof}
\end{comment}
\begin{proof}[Proof of Theorem \ref{LP2}]
It follows straight away from previous lemmas.
\end{proof}
\begin{proof}[Proof of Theorem \ref{TP2}] 
From Theorem \ref{LP2}, the problem $(P_{p,q})$ has a solution if, and only if, $1 > t_1=\lambda_1/\lambda_A$. Therefore, $(P_{p,q})$ has a solution if, and only if, $\lambda_A>\lambda_1$, concluding the proof.
\end{proof}
\section{Describing a local branch of solutions}

In this section, we will apply the local bifurcation results of Crandall and Rabinowitz \cite{CrandallRabinowitz} to study the behavior of the solutions obtained in the previous sections. As a consequence, we will obtain existence results for positive solutions for \( \lambda_A < \lambda_1 \) when suitable assumptions are satisfied.

We now rewrite the system   \eqref{P} as follows
$$
\left\{
\begin{array}{lcl}
-\Delta u+pu^{p-1}\vec{\alpha}(x)\cdot \nabla u&=&\left(a-\int_{\Omega}K_1(x,y)f(u,v)dy \right)u+bv\mbox{ in }\Omega,\\
-\Delta v+qv^{q-1}\vec{\beta}(x)\cdot \nabla v&=&\left(d-\int_{\Omega}K_2(x,y)g(u,v)dy \right)v+cu\mbox{ in }\Omega,\\
\qquad \qquad \qquad \qquad u,v&>&0\mbox{ in }\Omega,\\ \qquad \qquad \qquad \quad u=v&=&0\mbox{ on }\partial\Omega
\end{array}
\right.
$$
where $K_1,K_2\in\mathcal{K}$, $f$ and $g$ verify the conditions $(f_0)-(f_2)$ and $a,b,c,d>0$. Additionally, to obtain the results we set out to prove, we assume

\noindent$(f_3)$ $f,g:[0,\infty)\times[0,\infty)\rightarrow\R^+$ are $C^1-$functions.

Hence,  $\Phi,\Psi:C^2(\Omega)\times C^2(\Omega)\rightarrow C(\Omega)\times C(\Omega)$ given by $\Phi(u,v)=\Phi_{(u,v)}$ and $\Psi(u,v)=\Psi_{(u,v)}$ are $C^1-$functions. Indeed, it is enough to see that
\begin{equation}
\Phi_{u}(u,v)u_0=\int_{\Omega}K(x,y)f_u(u,v)u_0dy,\ \Phi_{v}(u,v)v_0=\int_{\Omega}K(x,y)f_v(u,v)v_0dy,\mbox{  with }u,v,u_0,v_0\in C^2(\Omega).
\end{equation}
Analogously,
\begin{equation}
\Psi_{u}(u,v)u_0=\int_{\Omega}K_2(x,y)g_u(u,v)u_0dy,\ \Psi_{v}(u,v)v_0=\int_{\Omega}K_2(x,y)g_v(u,v)v_0dy,\mbox{  with }u,v,u_0,v_0\in C^2(\Omega).
\end{equation}

The above problem can be rewritten, in a matrix form, by
$$
\left\{
\begin{array}{lcl}
-\Delta U&=&AU-H(U)\quad \mbox{ in }\Omega,\\
\quad U&>&0\quad\mbox{ in }\Omega,\\
\quad U&=&0\quad\mbox{ on }\partial\Omega,
\end{array}
\right.\leqno(P_U)
$$
where
\begin{equation*}
U=\left(
\begin{array}{c}
u\\v
\end{array}
\right),\quad A=\left(
\begin{array}{cc}
a & b\\c & d
\end{array}
\right),
\end{equation*}
and $H:C^2(\Omega)\times C^2(\Omega)\rightarrow C(\Omega)\times C(\Omega)$ is given by
\begin{equation*}
H(U)=H(u,v)=\left(
\begin{array}{c}
pu^{p-1}\vec{\alpha}\cdot \nabla u+\Phi_{(u,v)}(x)u\\
qv^{q-1}\vec{\beta}\cdot \nabla v+\Psi_{(u,v)}(x)v\\
\end{array}
\right)=\left(
\begin{array}{c}
h_1(u,v)\\h_2(u,v)
\end{array}
\right).
\end{equation*}
We notice that $H$ is a $C^1-$function satisfying
\begin{equation}
H'(U)V=\left[
\begin{array}{cc}
h_{1,u}(u,v) & 	h_{1,v}(u,v)\\
h_{2,u}(u,v) & 	h_{2,v}(u,v)\\
\end{array}
\right],
\end{equation}
where
\begin{eqnarray*}
h_{1,u}(u,v)u_0&=&p(p-1)u^{p-2}u_0\vec{\alpha}\cdot \nabla u+pu^{p-1}\vec{\alpha}\cdot \nabla u_0+\Phi_{u}(u,v)u_0u+\Phi_{(u,v)}uu_0,\\
h_{1,v}(u,v)v_0&=&\Phi_{v}(u,v)v_0u,\quad
h_{2,u}(u,v)u_0=\Psi_{u}(u,v)u_0v \mbox{ and}\\
h_{2,v}(u,v)v_0&=&q(q-1)v^{q-2}v_0\vec{\beta}\cdot \nabla v+qv^{q-2}\vec{\beta}\cdot \nabla v_0+\Psi_{v}(u,v)v_0v+\Psi_{(u,v)}vv_0.
\end{eqnarray*}

In order to use the Crandall-Rabinowitz theorem \cite{CrandallRabinowitz}, we will consider the auxiliary problem
$$
\left\{
\begin{array}{lcl}
-\Delta U&=&tAU-H(U)\mbox{ in }\Omega,\\
\qquad	U&=&0\mbox{ on }\partial\Omega.
\end{array}
\right.\leqno(P_U^t)
$$
Then, we  define $X:=(C^2(\Omega)\cap C_0(\overline{\Omega}))\times (C^2(\Omega)\cap C_0(\overline{\Omega}))$, $Y:=C(\Omega)\times C(\Omega)$, and also
\begin{equation}\label{eq1}
\begin{array}{cccl}
	\mathcal{F}:&\R\times X&\rightarrow&Y\\
	&(t,U)&\mapsto&\mathcal{F}(t,U):=-\Delta U-tAU+H(U).
\end{array}
\end{equation}

We observe that the operator $\mathcal{F}$ is well defined and continuous.  Moreover, its partial derivatives $\mathcal{F}_t$, $\mathcal{F}_U$ and $\mathcal{F}_{tU}$ exist and are continuous. In fact, for all $U,V\in X$, they are given by
\begin{equation}
	\mathcal{F}_t(t,U)=-AU,\quad \mathcal{F}_U(t,U)V=-\Delta V-tAV+H'(U)(V)\mbox{ \ and \ }\mathcal{F}_{tU}(t,U)V=-AV.
\end{equation}

We will establish some properties of $\mathcal{F}$, which are crucial in our approach. We begin considering the linear operator $\mathcal{F}_U(t_1,0)$, $t_1=\lambda_1/\lambda_A$, where $\lambda_1$ is the principal eigenvalue $(-\Delta,H_0^1(\Omega))$ and $\lambda_A>0$ its largest eigenvalue of $A$. Since $H'(0)=0$, we have that $\mathcal{F}_U(t_1,0)=-\Delta-t_1A$ is a bounded self-adjoint operator and $\mathcal{F}_U(t_1,0)V=0$ if, and only if, $-\Delta V=t_1AV$. Thus, from \cite[Section 3]{RM}, we get 
\begin{equation}
V=\left(
\begin{array}{c}
\xi_1\phi_1\\
\xi_2\phi_1
\end{array}
\right)=\phi_1\left(
\begin{array}{c}
	\xi_1\\
	\xi_2
\end{array}
\right)=\phi_1\xi,
\end{equation}
where $\phi_1>0$ is an eigenfunction of $(-\Delta,H_0^1(\Omega))$ associated with the eigenvalue $\lambda_1$ and $\xi=(\xi_1,\xi_2)$ is an eigenvector of $A$ associated with $\lambda_A$. So, as $\dim N[A-\lambda I]=1$, we have $N[\mathcal{F}_U(t_1,0)]=\mbox{span}\{V\}$, furthermore, as $\mathcal{F}_U(t_1,0)$ is self-adjoint, then $R[\mathcal{F}_U(t_1,0)]=\mbox{span}\{V\}^{\perp}$.

In order to use the Crandall-Rabinowitz theorem, we also need to prove a transversality condition given below.

\begin{lemma}[Transversality Condition]
Suppose that $A$ is a symmetric matrix, then $\mathcal{F}_{tU}(t_1,0)V\notin R[\mathcal{F}_U(t_1,0)]$.
\end{lemma}
\begin{proof}
Initially, see that $\mathcal{F}_{tU}(t_1,0)V=-AV$. Then, we will show that $-AV\notin R[\mathcal{F}_U(t_1,0)]$. Indeed, otherwise, there would be $W\in X$ such that $\mathcal{F}_U(t_1,0)W=-AV$, that is,
\begin{equation*}
-\Delta W-t_1AW=-AV.
\end{equation*}
Consequently, by applying the inner product and integrating, we get
\begin{equation*}
\int_{\Omega}(-\Delta W)\cdot V-t_1AW\cdot Vdx=-\int_{\Omega}-AV\cdot Vdx.
\end{equation*}
Provided that $A$ is symmetric, we also have
\begin{equation*}
\int_{\Omega}(-\Delta W)\cdot Vdx=\int_{\Omega}W\cdot (-\Delta V)dx=\int_{\Omega}W\cdot (t_1 AV)dx=t_1\int_{\Omega}AW\cdot V dx,
\end{equation*}
and so
\begin{equation*}
0=\int_{\Omega}AV\cdot Vdx=\lambda_A\int_{\Omega}(\xi_1^2+\xi_2^2)\phi_1dx,
\end{equation*}
which is,  a contradiction.
\end{proof}

Now, we are in the conditions of \cite[Theorem 1.7]{CrandallRabinowitz}, hence, we are able to derive the following result on local bifurcation, analogous to that found in \cite[Theorem 4.2]{BrownStavrakakis}.

\begin{theorem}
Assume that $K_1,K_2\in \mathcal{K}$ and $(f_0)-(f_3)$ hold with $p,q>1$. Let $A=\left(\begin{array}{cc}a & b\\b & d\end{array}\right)$ be a symmetric matrix with $a,b,c,d>0$ and $\lambda_A>0$ its largest eigenvalue. Then, there exist $\varepsilon_0>0$ and continuous functions $\eta:(-\varepsilon_0,\varepsilon_0)\rightarrow\R$ and $\rho:(-\varepsilon_0,\varepsilon_0)\rightarrow \mbox{span}\{V\}^{\perp}$. These functions satisfy  $\eta(0)=0$, $\rho(0)=0$ and every nontrivial solution to $\mathcal{F}(t,U)=0$ in a small neighborhood of $(t_1,0)$ is of the form $(t_\varepsilon,U_\varepsilon)=(t_1+\eta(\varepsilon),\varepsilon V+\varepsilon\rho(\varepsilon))$.
\end{theorem}

Returning to problem $(P_U^t)$, we get that
\begin{equation*}
\int_{\Omega}(-\Delta U_\varepsilon)\cdot Vdx=t_\varepsilon\int_{\Omega}AU_\varepsilon\cdot Vdx-\int_{\Omega}H(U_\varepsilon)\cdot Vdx.
\end{equation*}
Since,
\begin{equation*}
\int_{\Omega}(-\Delta U_\varepsilon)\cdot Vdx=\int_{\Omega}U_{\varepsilon}\cdot(-\Delta V) dx=\int_{\Omega}U_\varepsilon\cdot t_1AVdx,
\end{equation*}
 then, we have
\begin{equation*}
t_1\int_{\Omega}U_\varepsilon\cdot AVdx=t_\varepsilon\int_{\Omega}AU_\varepsilon\cdot Vdx-\int_{\Omega}H(U_\varepsilon)\cdot Vdx.
\end{equation*}
Therefore,
\begin{equation*}
t_1\int_{\Omega}U_\varepsilon\cdot AVdx=(t_1+\eta(\varepsilon))\int_{\Omega}AU_\varepsilon\cdot Vdx-\int_{\Omega}H(U_\varepsilon)\cdot Vdx.
\end{equation*}
Hence, since $A$ is a symmetric matrix, we get
\begin{equation*}
\eta(\varepsilon)=\frac{\int_{\Omega}H(U_\varepsilon)\cdot Vdx}{\int_{\Omega}AU_\varepsilon\cdot Vdx}.
\end{equation*}
Now, we will analyze the signs of the integrals. First, see that
\begin{eqnarray*}
H(U_\varepsilon)\cdot V&=&
\left[
\begin{array}{c}
pu_\varepsilon^{p-1}\vec{\alpha}(x)\cdot \nabla u_\varepsilon+\Phi_{(u_\varepsilon,v_\varepsilon)}(x)u_\varepsilon\\
qv_\varepsilon^{q-1}\vec{\beta}(x)\cdot \nabla v_\varepsilon+\Psi_{(u_\varepsilon,v_\varepsilon)}(x)v_\varepsilon\\	
\end{array}
\right]\cdot \left[
\begin{array}{c}
\xi_1\phi_1\\
\xi_2\phi_1
\end{array}
\right]\\
&=&[\xi_1(pu_\varepsilon^{p-1}\vec{\alpha}(x)\cdot \nabla u_\varepsilon+\Phi_{(u_\varepsilon,v_\varepsilon)}(x)u_\varepsilon)+\xi_2(qv_\varepsilon^{q-1}\vec{\beta}(x)\cdot \nabla v_\varepsilon+\Psi_{(u_\varepsilon,v_\varepsilon)}(x)v_\varepsilon)]\phi_1.
\end{eqnarray*}
It is important to note that $V=(\xi_1\phi_1,\xi_2\phi_1)$ and $\rho(\varepsilon)=(\rho_1(\varepsilon),\rho_2(\varepsilon))$, and so $U_\varepsilon=\varepsilon(\xi_1\phi_1+\rho_1(\varepsilon),\xi_2\phi_1+\rho_1(\varepsilon))$, that is, $U_{\varepsilon}=(u_\varepsilon,v_\varepsilon)$, where $u_\varepsilon=\varepsilon(\xi_1\phi_1+\rho_1(\varepsilon))$ and $v_\varepsilon=\varepsilon(\xi_2\phi_1+\rho_2(\varepsilon))$. 
On the other hand, by the homogeneity, it follows that
\begin{eqnarray*}
\int_{\Omega}\frac{H(U_\varepsilon)}{\varepsilon}\cdot Vdx&=\int_{\Omega}\xi_1\left(pu_\varepsilon^{p-1}\vec{\alpha}(x)\cdot \nabla \left(\frac{u_\varepsilon}{\varepsilon}\right) +\varepsilon^\gamma\Phi_{\left(\frac{u_\varepsilon}{\varepsilon},\frac{v_\varepsilon}{\varepsilon}\right)}(x)\frac{u_\varepsilon}{\varepsilon}\right)\phi_1 dx\\
&+\int_{\Omega}\xi_2\left(qv_\varepsilon^{q-1}\vec{\beta}(x)\cdot \nabla \left(\frac{v_\varepsilon}{\varepsilon}\right)+\varepsilon^\gamma\Psi_{\left(\frac{u_\varepsilon}{\varepsilon},\frac{v_\varepsilon}{\varepsilon}\right)}(x)\frac{v_\varepsilon}{\varepsilon}\right)\phi_1dx.
\end{eqnarray*}
In addition, as
\begin{equation*}
AU_\varepsilon\cdot V=A(\varepsilon V+\varepsilon\rho(\varepsilon))\cdot V=\varepsilon(AV\cdot V+A\rho(\varepsilon)\cdot V),
\end{equation*}
we get
\begin{equation}
\eta(\varepsilon)=\frac{\int_{\Omega}\left[\xi_1\left(pu_\varepsilon^{p-1}\vec{\alpha}(x)\cdot \nabla \left(\frac{u_\varepsilon}{\varepsilon}\right)+\varepsilon^\gamma\Phi_{\left(\frac{u_\varepsilon}{\varepsilon},\frac{v_\varepsilon}{\varepsilon}\right)}(x)\frac{u_\varepsilon}{\varepsilon}\right)+\xi_2\left(qv_\varepsilon^{q-1}\vec{\beta}(x)\cdot \nabla \left(\frac{v_\varepsilon}{\varepsilon}\right)+\varepsilon^\gamma\Psi_{\left(\frac{u_\varepsilon}{\varepsilon},\frac{v_\varepsilon}{\varepsilon}\right)}(x)\frac{v_\varepsilon}{\varepsilon}\right)\right]\phi_1dx}{\int_{\Omega}AV\cdot V+A\rho(\varepsilon)\cdot Vdx}.
\end{equation}
Taking $\delta: = \min\{\gamma, p-1, q-1\}$ we have

\begin{align*}
\dfrac{	\eta(\varepsilon)}{\varepsilon^{\delta}}=&\dfrac{\int_{\Omega}\xi_1\left(p\left(\frac{u_\varepsilon}{\varepsilon}\right)^{p-1}\varepsilon^{p-1-\delta }\vec{\alpha}(x)\cdot \nabla \left(\frac{u_\varepsilon}{\varepsilon}\right)+\varepsilon^{\gamma-\delta}\Phi_{\left(\frac{u_\varepsilon}{\varepsilon},\frac{v_\varepsilon}{\varepsilon}\right)}(x)\frac{u_\varepsilon}{\varepsilon}\right)\phi_1dx}{\int_{\Omega}AV\cdot V+A\rho(\varepsilon)\cdot Vdx}\\&+\dfrac{\int_\Omega\xi_2\left(q\left(\frac{v_\varepsilon}{\varepsilon}\right)^{q-1}\varepsilon^{q-1-\delta}\vec{\beta}(x)\cdot \nabla \left(\frac{v_\varepsilon}{\varepsilon}\right)+\varepsilon^{\gamma-\delta}\Psi_{\left(\frac{u_\varepsilon}{\varepsilon},\frac{v_\varepsilon}{\varepsilon}\right)}(x)\frac{v_\varepsilon}{\varepsilon}\right)\phi_1dx}{\int_{\Omega}AV\cdot V+A\rho(\varepsilon)\cdot Vdx}.
\end{align*}

Now, applying the Lebesgue's dominated convergence theorem, we consider all the cases involving the possible values of $\delta.$ 

If $\delta = \gamma < p-1, q-1$,  we obtain
\[\lim_{\varepsilon \to 0^+}\dfrac{	\eta(\varepsilon)}{\varepsilon^{\delta}}=\dfrac{\int_{\Omega}\left[\xi^2_1\Phi_{\left(V\right)}(x)+\xi^2_2\Psi_{\left(V\right)}(x)\right]\phi^2_1dx}{\int_{\Omega}AV\cdot Vdx}.\]
If $\delta = \gamma = p-1 < q-1$, we have
\[\lim_{\varepsilon \to 0^+}\dfrac{	\eta(\varepsilon)}{\varepsilon^{\delta}}=\dfrac{\int_{\Omega}\left[\xi^2_1\left(p\left(\xi_1\phi_1\right)^{p-2}\vec{\alpha}(x)\cdot \nabla \left(\xi_1\phi_1\right)+\Phi_{\left(V\right)}(x)\right)+\xi^2_2\Psi_{\left(V\right)}(x)\right]\phi^2_1dx}{\int_{\Omega}AV\cdot Vdx}.\]
If $\delta = \gamma = q-1 < p-1$, we get
\[\lim_{\varepsilon \to 0^+}\dfrac{	\eta(\varepsilon)}{\varepsilon^{\delta}}=\dfrac{\int_{\Omega}\left[\xi^2_1\Phi_{\left(V\right)}(x)+\xi^2_2\left(q\left(\xi_2\phi_1\right)^{q-2}\vec{\beta}(x)\cdot \nabla \left(\xi_2\phi_1\right)+ \Psi_{\left(V\right)}(x)\right)\right]\phi^2_1dx}{\int_{\Omega}AV\cdot Vdx}.\]
If $\delta = \gamma = q-1 = p-1$, we derive
\[\lim_{\varepsilon \to 0^+}\dfrac{	\eta(\varepsilon)}{\varepsilon^{\delta}}=\dfrac{\int_{\Omega}\left[\xi^2_1\left(p\left(\xi_1\phi_1\right)^{p-2}\vec{\alpha}\cdot \nabla \left(\xi_1\phi_1\right)+\Phi_{\left(V\right)}(x)\right)+\xi^2_2\left(q\left(\xi_2\phi_1\right)^{q-2}\vec{\beta}\cdot \nabla \left(\xi_2\phi_1\right)+ \Psi_{\left(V\right)}(x)\right)\right]\phi^2_1dx}{\int_{\Omega}AV\cdot Vdx}.\]
If $\delta = q-1 = p-1, \gamma $, we deduce
\[\lim_{\varepsilon \to 0^+}\dfrac{	\eta(\varepsilon)}{\varepsilon^{\delta}}=\dfrac{\int_{\Omega}\left[\xi^2_1p\left(\xi_1\phi_1\right)^{p-2}\vec{\alpha}\cdot \nabla \left(\xi_1\phi_1\right)+\xi^2_2q\left(\xi_2\phi_1\right)^{q-2}\vec{\beta}\cdot \nabla \left(\xi_2\phi_1\right)\right]\phi^2_1dx}{\int_{\Omega}AV\cdot Vdx}.\]
If $\delta = q-1 < p-1< \gamma $, we obtain
\[\lim_{\varepsilon \to 0^+}\dfrac{	\eta(\varepsilon)}{\varepsilon^{\delta}}=\dfrac{\int_{\Omega}q\left(\xi_2\phi_1\right)^{q}\vec{\beta}\cdot \nabla \left(\xi_2\phi_1\right)dx}{\int_{\Omega}AV\cdot Vdx} = q\xi_2^{q+1}(q+1)^{-1}\frac{\int_\Omega \phi_1^{q+1} \operatorname{div} \vec{\beta}}{ \int_{\Omega}AV\cdot Vdx}.\]
If $\delta = p-1 <q-1, \gamma $, we get
\[\lim_{\varepsilon \to 0^+}\dfrac{	\eta(\varepsilon)}{\varepsilon^{\delta}}=\dfrac{\int_{\Omega}p\left(\xi_1\phi_1\right)^{p}\vec{\alpha}\cdot \nabla \left(\xi_1\phi_1\right)dx}{\int_{\Omega}AV\cdot Vdx}= p\xi_1^{p+1}(p+1)^{-1}\frac{\int_\Omega \phi_1^{p+1} \operatorname{div} \vec{\alpha}}{ \int_{\Omega}AV\cdot Vdx}.\]

According to the sign of each limiting case presented above, we can determine the type of bifurcation direction when $\varepsilon\approx0^+$. This means that, if $\eta(\varepsilon)>0$ when $\varepsilon\approx0^+$, we have a supercritical bifurcation in $(t_1,0)$, on the other hand if $\eta(\varepsilon)<0$ when $\varepsilon\approx0^+$, we have a subcritical bifurcation in $(t_1,0)$. In particular, we obtain the following result:

\begin{corollary}
    Assume that $K_1,K_2\in \mathcal{K}$ and $(f_0)-(f_3)$ hold with $p,q>1$.
    \begin{enumerate}
        \item[i)] If $q<\min\{p,\gamma+1\}$ and
        $$\int_\Omega \phi_1^{q+1} \operatorname{div} \vec{\beta}<0,$$
        then there exists $\epsilon>0$ such that \eqref{P} has a positive solution if $\lambda_A<\lambda_1$  and $\lambda_1<(1+\epsilon)\lambda_A$.
        \item[ii)] If $p<\min\{q,\gamma+1\}$ and
        $$\int_\Omega \phi_1^{p+1} \operatorname{div} \vec{\alpha}<0,$$
        then there exists $\epsilon>0$ such that \eqref{P} has a positive solution if $\lambda_A<\lambda_1$  and $\lambda_1<(1+\epsilon)\lambda_A$.
    \end{enumerate}
\end{corollary}

\section{Declarations}

\noindent {\bf Ethical Approval:} \,\,\,\, Not applicable \\

\noindent {\bf Competing interests:}  \,\,\, The authors declare that they have no conflict of interest. \\

\noindent {\bf Authors' contributions:} \,\,\, The authors wrote the article together. \\

\noindent {\bf Funding: }  Romildo Lima was partially supported by CNPq/Brazil 306.411/2022-9 and 176.596/2023-2. Willian Cintra was partially supported by FAPDF 00193.00001821/2022-21 and CNPq 310664/2023-3.\\

\noindent {\bf Availability of data and materials: } \,\,\, Not applicable


\newpage

\begin{thebibliography}{1} 

\bibitem{Allegretto-Nistri} W. Allegretto and P. Nistri, \textit{On a class of nonlocal problems with applications to mathematical biology. Differential equations with applications to biology},(Halifax, NS, 1997), 1-14, Fields Inst. Commun., \textbf{21}, Am. Math. Soc., Providence, RI (1999).

\bibitem{Alves-Delgado-Souto-Suarez} C. O. Alves, M. Delgado, M. A. S. Souto and A. Suárez, \textit{Existence of positive solution to a nonlocal logistic population model}, Z. Angew. Math. Phys. 66 (2015), 943-953. 

%\bibitem{Amann}H. Amann,\textit{ Existence and multiplicity theorems for semilinear elliptic boundary value problems}, Math. Z. 150 (1976) 281–295.

\bibitem{BelgacemCosner}F. Belgacem, C. Cosner, \textit{The effects of dispersal along environmental gradients on the dynamics of populations in
heterogeneous environments}, Can. Appl. Math. Q. 3 (1995) 379–397.

\bibitem{BrownStavrakakis} K. J. Brown and N. M. Stavrakakis, {\it Global bifurcation results for a semilinear elliptic equation on all of $\R^N$}, Duke Math. J. \textbf{85} (1996), 77-94.

\bibitem{CantrellCosner} R.S. Cantrell, C. Cosner, \textit{Spatial ecology via reaction–diffusion equations}, Wiley Series in Mathematical and
Computational Biology, John Wiley \& Sons, 2003.

\bibitem{CantrellCosner2} R.S. Cantrell, C. Cosner, Diffusive logistic equations with indefinite weights: Population models in disrupted
environments. II, SIAM J. Math. Anal. 22 (1991) 1043–1064.

\bibitem{Chen-Shi} S. Chen and J. Shi, \textit{Stability and Hopf bifurcation in a diffusive logistic population model with nonlocal delay effect}, J. Differential Equations, \textbf{253}, (2012) 3440-3470.

\bibitem{Chipot} M. Chipot, \textit{Remarks on Some Class of Nonlocal Elliptic Problems}, Recent Advances on Elliptic and Parabolic Issues,World Scientific, (2006) 79-102.

\bibitem{CintraMontenegroSuarez} W. Cintra, M. Montenegro and A. Suárez, \textit{The logistic equation with nonlinear advection term}, Nonlinear Anal. RWA 65 (2022) 103503.

\bibitem{CMRS2019} W. Cintra, C. Morales-Rodrigo and A. Suárez, \textit{Unilateral global bifurcation for a class of quasilinear elliptic systems and applications}, J. Differential Equations, 267 (2019), 619-657.

\bibitem{CMS}  W. Cintra, C. Morales-Rodrigo and A. Suárez, \textit{Coexistence states in a cross-diffusion system of a predator-prey model with predator satiation term.} Math. Models Methods Appl. Sci. 28 (2018), nº. 11, 2131–2159.

\bibitem{Correa-Delgado-Suarez} F. J. S. A. Corrêa, M. Delgado and A. Suárez, \textit{Some nonlinear heterogeneous problems with nonlocal reaction term}, Advances in Differential Equations, \textbf{16}, (2011) 623-641.

\bibitem{CS1}  F.J.S.A. Corrêa, M.A.S. Souto, On maximum principles for cooperative elliptic systems via fixed point index, Nonlinear
Anal. {\bf 26} (1997) 997–1006.



\bibitem{Cosner} C. Cosner, \textit{Reaction–diffusion–advection models for the effects and evolution of dispersal}, Discrete Contin. Dyn. Syst. 34 (2014) 1701–1745.

\bibitem{Coville} J. Coville, \textit{Convergence to equilibrium for positive solutions of some mutation-selection model}. 2013 .hal-00855334

\bibitem{CrandallRabinowitz} M. Crandall and P. H. Rabinowitz, \textit{Bifurcation from simple eigenvalues}, J. Funct. Anal., \textbf{8}, (1971) 321-340.

\bibitem{R} R. N. de Lima, R. C. Duarte and M. A. S. Souto,\textit{A nonlocal logistic equation with nonlinear advection term}, preprint

\bibitem{RM} R. N. de Lima and M. A. S. Souto, \textit{Existence of positive solution to a system of elliptic equations via bifurcation theory}, J. Math.Anal.Appl.457(2018)287–304.

%\bibitem{DDS} M. Delgado, I. B. M. Duarte, and A. Suárez, \textit{Nonlocal singular elliptic system arising from the amoeba–bacteria population dynamics}  Commun. Contemp. Math. Vol. 21, No. 07, 1850051 (2019).

\bibitem{Gilbarg-Trudinger} D. Gilbarg and N. S. Trudinger, \textit{Elliptic Partial Differential Equations of Second Order.}Berlin: Springer-Verlag, 1998.

%\bibitem{Leman-Meleard-Mirrahimi} H. Leman, S. Méléard and S. Mirrahimi, \textit{Influence of a spatial structure on the long time behavior of a competitive Lotka-Volterra type system},  Discrete Contin. Dyn. Syst. Ser. B 20 (2015), \textbf{2}, 469-493.

\bibitem{LopezGomes} J. López-Gómez, \textit{Linear Second Order Elliptic Operators}, World Scientific, Singapore, 2013.

%\bibitem{Pao} C. V. Pao, \textit{Eigenvalue problems of a degenerate quasilinear elliptic equation}, Rocky Mountain J. Math. 40 (2010) 305–311.

\bibitem{Rabinowitz} Rabinowitz, P.: \textit{Some global results for nonlinear eigenvalue problems}. J. Funct. Anal. 7, 487–513 (1971)

\bibitem{SW} J. Shi, X. Wang, \textit{On global bifurcation for quasilinear elliptic systems on bounded domains}, J. Differential Equations 246(7) (2009) 2788–2812.

\bibitem{Sun-Shi-Wang} L. Sun, J. Shi and Y. Wang, \textit{Existence and uniqueness of steady state solutions of a nonlocal diffusive logistic equation}, Z. Angew. Math. Phys., \textbf{64}, (2013) 1267-1278. 

\bibitem{U} K. Umezu, \textit{Logistic elliptic equation with a nonlinear boundary condition arising from coastal fishery harvesting}, Nonlinear Anal. RWA 70 (2023) 103788.
\end{thebibliography}
\end{document}